\documentclass[a4paper,10pt]{article}
\usepackage{psfrag}
\usepackage{amsmath, amsfonts, amscd, amssymb, amsthm, enumerate}
\usepackage[all]{xy}
\usepackage[francais,english]{babel}
\usepackage[latin1]{inputenc}

\title{The geometric realization of a simplicial Hausdorff space is Hausdorff}
\author{Cl\'ement de Seguins-Pazzis\footnote{e-mail adress: dsp.prof@gmail.com}
\footnote{This work
was completed while the author was working on his PhD thesis at the Institut Galil\'ee in Universit\'e Paris Nord,
99 avenue Jean-Baptiste Cl\'ement,
93430 Villetaneuse, FRANCE}}
\date{\today}

\def\defterm{\textbf}
\def\N{\mathbb{N}}

\def\calJ{\mathcal{J}}
\def\calP{\mathcal{P}}

\newcommand{\Hom}{\operatorname{Hom}}

\newcommand{\id}{\operatorname{id}}

\newcommand{\card}{\#}

\newcommand{\red}{\operatorname{red}}

\renewcommand{\epsilon}{\varepsilon}
\renewcommand{\setminus}{\smallsetminus}

\newtheorem{theo}{Theorem}[section]
\newtheorem{prop}[theo]{Proposition}
\newtheorem{cor}[theo]{Corollary}
\newtheorem{lemme}[theo]{Lemma}
\newtheorem{Def}{Definition}[section]
\newtheorem{Not}[Def]{Notation}

\theoremstyle{remark}
\newtheorem{Rems}{Remarks}
\newtheorem{Rem}[Rems]{Remark}
\newtheorem{ex}{Example}

\begin{document}

\maketitle

\begin{abstract}
It is shown that the thin geometric realization of a simplicial Hausdorff space is Hausdorff.
This proves a famous claim by Graeme Segal that the thin geometric realisation of a simplicial
k-space is a k-space.
\end{abstract}



\section{Introduction}

\subsection{The main problem}
In one of his many landmark papers \cite{Segal-class}, Graeme Segal introduced the geometric
realization functor for simplicial spaces, which he called the
``thin'' realization functor in the subsequent article \cite{Segal-cat}. He claimed that the thin geometric realization of a
simplicial space which is compactly-generated Hausdorff degreewise
must be compactly-generated Hausdorff. However, while it is essentially obvious
that the geometric realisation of a simplicial compactly-generated space must be compactly-generated,
the Hausdorff property is a whole different matter since cocartesian squares are implicit in the definition of the thin
geometric realization and they are known to behave badly with respect to separation axioms.
At the time of \cite{Segal-class}, Segal's claim was thought to be dubious and
no convicing proof of it ever appeared in the litterature. This difficulty brought
some to turn away from k-spaces and work with weak-Hausdorff compactly-generated spaces instead. It is much easier indeed to show that
the geometric realization of a compactly-generated weak-Hausdorff space is itself compactly-generated weak-Hausdorff
(this can be done with little effort using the tools from Appendix A of \cite{GaunceLewis}). 
In the following pages, we will prove that Segal was right after all!

\subsection{Definitions and notation}

In this paper, we will use the French notation for the sets of integers:
$\N$ will denote the set of natural numbers (i.e.\ non-negative integers), and $\N^*$ the one of positive integers.
Recall the simplicial category $\Delta$ whose objects are the ordered sets $[n]=\{0,1,\dots,n\}$
for $n\in \mathbb{N}$ and whose morphisms are the non-decreasing maps, with the obvious compositions and identities.
All the morphisms are composites of morphisms of two types, namely the face morphisms
$$\delta_i^k : \begin{cases}
[k] & \rightarrow [k+1] \\
j<i & \mapsto j \\
j\geq i & \mapsto j+1
\end{cases}$$ for $k\in \N$ and $i\in [k+1]$, and the degeneracy
morphisms
$$\sigma_i^k : \begin{cases}
[k] & \rightarrow [k-1] \\
j\leq i & \mapsto j \\
j> i & \mapsto j-1
\end{cases}$$ for $k\in \N^*$ and $i\in [k-1]$.
See \cite{May} for a comprehensive account.

There is (covariant) functor $\Delta^*:\Delta \rightarrow \text{Top}$
which sends $[n]$ to the $n$-simplex $\Delta^n:=\Bigl\{(t_i)_{0\leq i\leq n}\in \mathbb{R}_+^{n+1} :
\underset{i=0}{\overset{n}{\sum}}t_i=1\Bigr\}$, and any morphism
$\delta: [n] \rightarrow [m]$ to
$$\Delta^*(\delta) :
\begin{cases}
\Delta^n & \rightarrow \Delta^m \\
(t_i)_{0\leq i \leq n} & \mapsto
\biggl(\underset{i\in \delta^{-1}(j)}{\sum}t_i\biggr)_{0\leq j\leq m.}
\end{cases}$$

A \textbf{simplicial space} is a contravariant functor $\Delta \rightarrow \text{Top}$.
Given such a functor, we set $A_n:=A([n])$ for any $n\in \mathbb{N}$.
For $k\in \mathbb{N}$, we will write $d_i^k:=A(\delta_i^k)$ for $i\in [k+1]$
(the face maps of $A$), and
$s_i^k:=A(\sigma_i^k)$ for $i\in [k-1]$ (the degeneracy maps of $A$).
When no confusion is possible, we will simply write $\delta_i$ instead of
$\delta_i^k$, $\sigma_i$ instead of
$\sigma_i^k$, $d_i$ instead of $d_i^k$ and $s_i$ instead of $s_i^k$.
If $\delta$ is a morphism in $\Delta$, we will also write $\delta^*$ instead of
$A(\delta)$.

For $n\in \mathbb{N}$, a point $x\in A_n$ is said to be \defterm{degenerate}
when in the image of some $s_i$.

\begin{Def}
The \textbf{thin geometric realization} of a simplicial space $A$, denoted by $|A|$, is 
the quotient space of $\underset{n\in \mathbb{N}}{\coprod}A_n \times \Delta^n$
under the relations
$(x,\Delta^*(\delta)[y])\sim (A(\delta)[x],y)$, for $(m,n)\in \mathbb{N}^2$,
$x\in A_m$, $y\in \Delta_n$ and $\delta \in \Hom_{\Delta}([n],[m])$.
\end{Def}

For every $n\in \mathbb{N}$, we thus have a natural map 
$$\pi_n: A_n \times \Delta^n \rightarrow |A|.$$

\begin{Rem}
Since no homotopy group will be considered here, no confusion should be excepted from our using the notation $\pi_n$
to designate the above map. 
\end{Rem}

\paragraph{}Our simple aim here is to prove the following theorem:

\begin{theo}\label{main}
Let $A$ be a simplicial space and assume that $A_n$ is Hausdorff for each $n \in \N$.
Then, $|A|$ is Hausdorff.
\end{theo}

The proof, although very technical, has a very straightforward basic strategy:
we will give a general construction of ``flexible" open neighborhoods for the points of $|A|$
(see Section \ref{gencon} for the construction and Section \ref{opensec} for the proof of openness),
and then show that those neighborhoods may be used to separate points
(Section \ref{sepsec}). In the rest of the paper, $A$ denotes an arbitrary simplicial space
(no separation assumption will be made until Section \ref{sepsec}).

\subsection{List of notation}

\bigskip

\def\8[#1;#2]{\hbox to \textwidth{#1 \hfill p. #2}\vskip2truemm}
\8[$(x,\alpha)$ \quad ($x\in A_n$ a non-degenerate simplex, $\alpha \in \Delta^n \setminus \partial \Delta^n$);\pageref{gencon}]
\8[$U(\sigma)$, \quad ($\sigma : [k] \twoheadrightarrow [n]$);\pageref{definitionofUparenthesesigma}]
\8[$\Gamma'$;\pageref{categoryGamma'}]
\8[$f : [k] \Rightarrow [k']$, a morphism in $\Gamma'$;\pageref{categoryGamma'}]
\8[$\red(f)$, $\sup(f)$, \quad ($f$ a morphism in $\Gamma'$) ;\pageref{redetsup}]
\8[$t_i(\alpha)$, \quad ($\alpha \in \Delta^n \setminus \partial \Delta_n$);\pageref{definitionofW}]
\8[$W(f,\epsilon)$, \quad ($f : [n] \Rightarrow [k]$ and $\epsilon \in \,\left]0,1\right[$);\pageref{definitionofW}]
\8[$\card_f(i)$, \quad ($f : [k] \Rightarrow [k']$, $i \in [k']$);\pageref{orderingsection}]
\8[$f^{+i}$, \quad ($f : [k] \Rightarrow [k']$, $i\in [k']$);\pageref{fplusi}]
\8[$f \leq g$, $f \subset g$, \quad ($f$ and $g$ two morphisms in $\Gamma'$);\pageref{orderdefinition}]
\8[$U_{k,\epsilon}$, $U_\epsilon$;\pageref{uepsdef}]
\8[$f_{-i}$, \quad ($f : [k] \Rightarrow [k']$, $i\in [k']$);\pageref{fmoinsi}]

\section{Constructing open subsets in a geometric realization}\label{gencon}

In the whole section, we fix an integer $n \in \N$,
a non-degenerate simplex $x \in A_n$ and a point $\alpha \in \Delta^n \setminus \partial \Delta^n$.
Our goal is to give a general construction of non-trivial open neighborhoods of $\pi_n(x,\alpha)$ in $|A|$.
More specifically, we shall construct a whole family $(U_\epsilon)_{\epsilon \in \left]0,1\right[}$
of open neighborhoods of $\pi_n(x,\alpha)$ in $|A|$.  

A basic idea is to construct, for every $k \in \N$, an open subset
$V_k$ of $A_k \times \Delta^k$ such that $(x,\alpha) \in V_n$ and
the family $(V_k)_{k \in \N}$ is \defterm{compatible}, i.e.\
for every morphism $f : [k] \rightarrow [k']$ in $\Delta$:
$$\forall (y,\beta)\in A_{k'} \times \Delta^k, \; (y,f_*(\beta)) \in V_{k'}
\Leftrightarrow (f^*(y),\beta) \in V_{k}\; ;$$
in this case, $\underset{k \in \N}{\bigcup}\,\pi_k(V_k)$ is a subset of $|A|$ that contains
$\pi_n(x,\alpha)$, and its inverse image by $\pi_k$ is $V_k$ for each $k \in \N$; from the very definition of the topology on $|A|$, 
it follows that $\underset{k \in \N}{\bigcup}\,\pi_k(V_k)$ is an open subset of $|A|$. 

\begin{Rem}
Since every morphism in the simplicial category is a composite of face and degeneracy morphisms,
a family $(V_k)_{k \in \N}$ is compatible if and only if it satisfies the following two sets of properties:
\begin{equation}\label{faceprop}
\forall k \in \N^*, \; \forall (y,\beta) \in A_k \times \Delta^{k-1}, \; \forall i \in [k], \;
(y,(\delta_i)_*(\beta)) \in V_k \, \Leftrightarrow \, (d_i(y),\beta) \in V_{k-1.}
\end{equation}
\begin{equation}\label{degprop}
\forall k \in \N, \; \forall (y,\beta) \in A_k \times \Delta^{k+1}, \; \forall i \in [k], \;
(y,(\sigma_i)_*(\beta)) \in V_k \, \Leftrightarrow \, (s_i(y),\beta) \in V_{k+1.}
\end{equation}
\end{Rem}

\subsection{Suitable families of open subsets of the $A_k$'s}\label{suitfam}

Our starting point is the following very basic lemma on simplicial sets:

\begin{lemme}\label{degenlemma}
Let $x \in A_n$ be a non-degenerate simplex.
Let $\sigma : [N] \twoheadrightarrow [n]$ and $\tau : [N] \twoheadrightarrow [m]$ be epimorphisms. Then, the following conditions are equivalent:
\begin{enumerate}[(i)]
\item The simplex $\sigma^*(x)$ belongs to $\tau^*(A_m)$.
\item There is an epimorphism $\rho : [m] \twoheadrightarrow [n]$ such that $\sigma=\rho \circ \tau$.
\item $\forall (i,j)\in [N]^2, \; \tau(i)=\tau(j) \Rightarrow \sigma(i)=\sigma(j)$.
\end{enumerate}
\end{lemme}

\begin{proof}
The only non-trivial statement is that (i) implies (ii). Assume then that $\sigma^*(x)=\tau^*(y)$ for some
$y \in A_m$. We may choose a section $\delta : [n] \hookrightarrow [N]$ of $\sigma$, whence
$(\tau \circ \delta)^*(y)=(\sigma \circ \delta)^*(x)=x$. If $\tau \circ \delta$ were not one-to-one, we would be able
to decompose it as $\tau \circ \delta=\sigma' \circ s_i^n$ for some $i$ and some morphism $\sigma'$,
which would yield $x \in s_i(A_{n-1})$, contradicting the fact that $x$ is non-degenerate.
\end{proof}

\begin{Def}
Let $N \geq n$ be an integer and $x \in A_n$ be a non-degenerate simplex.
A family $(U_\sigma)_{\sigma : [N]\twoheadrightarrow [n]}$
is called \defterm{$x$-admissible} when:
\begin{enumerate}[(i)]
\item The set $U_\sigma$ is an open neighborhood of $\sigma^*(x)$ in $A_N$
for every epimorphism $\sigma : [N]\twoheadrightarrow [n]$.
\item For every $\sigma : [N]\twoheadrightarrow [n]$ and every
$\tau : [N]\twoheadrightarrow [m]$, one has
$$U_\sigma \cap \tau^*(A_m) \neq \emptyset \,\Leftrightarrow\, (\exists \rho : [m]
\twoheadrightarrow [n] :\; \sigma =\rho \circ \tau).$$
\end{enumerate}
\end{Def}

\begin{Rem}
Let $\sigma : [N] \twoheadrightarrow [k]$ and $\tau : [N]\twoheadrightarrow [m]$.
By the previous lemma, $\sigma_*(x)$ does not belong to the union of all
$\tau^*(A_m)$, for $N \geq m$, where $\tau$ ranges over the set of all epimorphisms from $[N]$ for which no epimorphism
$\rho$ satisfies $\sigma =\rho \circ \tau$. Assuming that, for every $\sigma$, we may find an open neighborhood
$U_\sigma$ of $\sigma^*(x)$ which is disjoint from this union, then the family $(U_\sigma)_{\sigma : [N]\twoheadrightarrow [n]}$ is obviously $x$-admissible.
\end{Rem}

\begin{ex}\label{exvois}
Assume that $A_k$ is Hausdorff for every $k \in \N$. Let $\sigma : [N] \twoheadrightarrow [k]$.
Then, for every epimorphism $\tau : [N]\twoheadrightarrow [m]$ such that no $\rho$ satisfies $\sigma =\rho \circ \tau$,
the subset $\tau^*(A_m)$ is closed in $A_N$ since it is a retract of $A_N$, whence the (finite) union of all such subsets
is closed in $A_N$ and we may choose $U_\sigma$ as its complementary subset in $A_N$ (and any open neighborhood of $\sigma^*(x)$
in this complementary subset will also do).
\end{ex}

\paragraph{}
Throughout the rest of the section, we set an $x$-admissible family $(U_\sigma)_{\sigma : [N] \twoheadrightarrow [n]}$
which we first extend as follows: given $k \leq N$ and $\sigma : [k] \twoheadrightarrow [n]$,
we set 
\label{definitionofUindicesigma}
$$U_{\sigma}:=\underset{\tau : [N] \twoheadrightarrow [k]}{\bigcap} (\tau^*)^{-1}(U_{\sigma \circ \tau}) \subset
A_k.$$
The following properties then generalize the axioms defining an $x$-admissible family:

\begin{lemme}\label{firstproperties}
Let $\sigma : [k] \twoheadrightarrow [n]$ with $k \leq N$. Then,
\begin{enumerate}[(i)]
\item One has $\sigma ^*(x) \in U_\sigma$.
\item For every $\sigma' : [k'] \twoheadrightarrow [k]$ with $k' \leq N$, one has
$(\sigma')^*(U_\sigma) \subset U_{\sigma \circ \sigma'}$.
\item For every $\tau : [k] \twoheadrightarrow [k']$, the condition $U_\sigma \cap \tau^*(A_{k'}) \neq \emptyset$
is equivalent to $\exists \rho: [k'] \twoheadrightarrow [n]: \; \rho \circ \tau =\sigma$.
\end{enumerate}
\end{lemme}

\begin{proof}
\begin{enumerate}[(i)]
\item Indeed $\tau^*(\sigma^*(x))=(\sigma \circ \tau)^*(x) \in U_{\sigma \circ \tau}$ for every $\tau : [N] \twoheadrightarrow [k]$.
\item Let $x' \in U_{\sigma }$ and $\tau: [N]
\twoheadrightarrow [k']$.
Then $\tau^*((\sigma')^*(x'))=(\sigma' \circ \tau)^*(x')
\in U_{\sigma \circ (\sigma' \circ \tau)}
=U_{(\sigma \circ \sigma') \circ \tau}$. Thus $(\sigma')^*(x')
\in U_{\sigma \circ \sigma'.}$
\item Let $x \in A_{k'}$
such that $\tau^*(x) \in U_\sigma$. Choose an arbitrary $\sigma': [N]
  \twoheadrightarrow [k]$. Then $(\tau \circ \sigma')^*(x)=(\sigma')^*(\tau^*(x)) \in U_{\sigma \circ \sigma'}$.
It follows from axiom (ii) that some $\rho : [k'] \twoheadrightarrow [n]$ satisfies
$\sigma \circ \sigma'=\rho \circ \tau \circ \sigma'$, whence $\sigma =\rho \circ \tau$ since $\sigma'$ is onto. The converse is trivial.
\end{enumerate}
\end{proof}

\subsection{The open subsets $U(\sigma)$}\label{upar}

\begin{Def}
For $\sigma : [k] \twoheadrightarrow [n]$, set
$$I_\sigma:=\Bigl\{\delta : [k'] \hookrightarrow [k] \quad
  \text{s.t.} \quad \sigma \circ \delta :
[k'] \twoheadrightarrow [n] \quad \text{and} \quad k' \leq N \Bigl\}$$
and\label{definitionofUparenthesesigma}
$$U(\sigma):=\underset{\delta \in I_\sigma}{\bigcap}(\delta^*)^{-1}(U_{\sigma \circ \delta}) \subset A_k.$$
\end{Def}
Clearly, $I_\sigma$ is non-empty and, better still, for every $(i,j) \in [k]^2$ such that $\sigma(i) \neq \sigma(j)$, there is
some $\delta \in I_\sigma$ with $i$ and $j$ in its range.

\vskip 2mm
\noindent The $U(\sigma)$ sets have the following main properties:

\begin{prop}\label{uproperties}
Let $\sigma : [k] \twoheadrightarrow [n]$. Then,
\begin{enumerate}[(a)]
\item The set $U(\sigma)$ is an open neighborhood of $\sigma^*(x)$ in $A_k$.
\item One has $U(\sigma) \subset U_\sigma$ whenever $k \leq N$.
\item For every $\delta: [i] \hookrightarrow [k]$ such that $\sigma \circ \delta$ is onto, one has
$\delta^*(U(\sigma)) \subset U(\sigma \circ \delta)$.
\item For every $\tau: [k'] \twoheadrightarrow [k]$, one has $\tau^*(U(\sigma)) \subset U(\sigma \circ \tau)$.
\item For every $\tau: [k] \twoheadrightarrow [k']$, the condition
$\tau^*(A_{k'}) \cap U(\sigma) \neq \emptyset$ is equivalent to
the existence of some $\rho : [k'] \twoheadrightarrow [n]$ such that $\sigma=\rho \circ \tau$.
\end{enumerate}
\end{prop}

\begin{proof}
\begin{enumerate}[(a)]
\item trivially derives from statement (i) in Lemma \ref{firstproperties} and the definition of $I_\sigma$.
\item is obvious since $\id_{[k]} \in I_\sigma$ whenever $k \leq N$.
\item Let $x' \in U(\sigma)$ and $\delta' \in I_{\sigma \circ \delta}$. Then $\delta \circ \delta' \in I_\sigma$ and so
$(\delta')^*(\delta^*(x))=(\delta \circ \delta')^*(x') \in U(\sigma \circ \delta \circ \delta')$.
Thus $\delta^*(x) \in U(\sigma \circ \delta)$.
\item Let $x \in U(\sigma)$. Let $\delta : [i] \hookrightarrow [k']$ in $I_{\sigma \circ \tau}.$
We decompose $\tau \circ \delta=\delta' \circ \tau'$ where $\delta'$ is a monomorphism and $\tau'$ an epimorphism.
Since $\sigma \circ \tau \circ \delta=\sigma \circ \delta' \circ \tau'$ is an epimorphism, $\sigma \circ \delta'$
also is, whence $(\delta')^*(x) \in U_{\sigma \circ \delta'}$
(notice that the domain of $\delta'$ is $[j]$ for some $j \leq i \leq N$ since $\tau'$ is onto). By statement (iii) in Lemma \ref{firstproperties},
it follows that $(\tau')^*\bigl((\delta')^*(x)\bigr) \in U_{\sigma \circ \delta' \circ \tau'}$, 
whence $\delta^*\bigl(\tau^*(x)\bigr) \in U_{\sigma \circ \tau \circ \delta.}$
Therefore, $\tau^*(x) \in U(\sigma \circ \tau)$.

\item If $\sigma=\rho \circ \tau$ for some $\rho : [k'] \rightarrow [n]$, then
$\sigma^*(x)=\tau^*(\rho^*(x)) \in \tau^*(A_{k'}) \cap U(\sigma)$. \\
Assume conversely that there is some $x \in A_{k'}$ such that $\tau^*(x) \in U(\sigma)$.
Let $\delta \in I_\sigma$ with domain $[i]$. Then $(\tau \circ \delta)^*(x) \in U_{\sigma \circ \delta.}$
Then, let us decompose $\tau \circ \delta=\delta' \circ \tau'$, where $\delta'$ is a monomorphism and $\tau'$ an epimorphism.
Then, $(\tau')^*((\delta')^*(x)) \in U_{\sigma \circ \delta}$, whence statement (iii) in Lemma \ref{firstproperties}
shows that $\forall (y,z)\in [i]^2, \; \tau'(y)=\tau'(z) \Rightarrow (\sigma \circ \delta)(y)=(\sigma \circ \delta)(z)$.
Since $\delta'$ is one-to-one, this yields $\tau(y')=\tau(z') \Rightarrow \sigma(y')=\sigma(z')$ for every $y'$ and $z'$ in the range of $\delta$.
Let finally $(y,z) \in [k]^2$ be such that $\sigma(y) \neq \sigma(z)$. By a previous remark, there is some $\delta \in I_\sigma$
the range of which contains $y$ and $z$. Hence, $\tau(y) \neq \tau(z)$.
\end{enumerate}
\end{proof}

\subsection{The category $\Gamma'$}\label{categoryGamma'}

The relation $\leq$ on $\calP(\N)$ defined by
$$A \leq B\; \overset{\text{def}}{\Leftrightarrow}\; \bigl(A=B \quad \text{or} \quad \sup A < \inf B\bigr)$$
yields a structure of poset\footnote{Notice that $\emptyset$ is the minimum element of $\calP(\N)$ for $\leq$
with the usual convention that $\sup \emptyset=-\infty$ and $\inf \emptyset=+\infty$.}
on $\calP(\N)$. We define the category $\Gamma '$ as the one with the same objects as $\Delta$
and for which, for any $(k,k')\in \N^2$, the morphisms from $[k]$ to $[k']$
are the increasing maps $\mathcal{P}([k]) \rightarrow \mathcal{P}([k']) $
which respect disjoint unions and map non-empty sets to non-empty sets, with the obvious composition of morphisms.
To avoid any confusion with the simplicial category, a morphism $f$ from $[k]$ to $[k']$ in $\Gamma'$
will be written $f : [k] \Rightarrow [k']$.

A morphism $f : [k] \Rightarrow [k']$ in $\Gamma'$ is called \textbf{onto} when $f([k])=[k']$.
To every such morphism corresponds an epimorphism $\sigma : [k'] \twoheadrightarrow [k]$ in $\Delta$
defined by: $\forall i \in [k'], \; i \in f \big\{\sigma(i)\bigr\}$.
Conversely, to every epimorphism $f : [k'] \twoheadrightarrow [k]$ in $\Gamma$ corresponds a unique
$\Gamma '(f): [k] \Rightarrow [k']$ defined by $\forall A \in \calP([k]), \; \Gamma '(f)(A)=f^{-1}(A)$.
Clearly, we have just defined reciprocal bijections between the set of epimorphisms from $[k']$ to $[k]$ in $\Delta$
and the set of onto morphisms from $[k]$ to $[k']$ in $\Gamma '$.

\begin{Not}\label{redetsup}
Let $f:[k] \Rightarrow [k']$ and $\delta: [k'] \hookrightarrow [k'']$.
Define then $\delta _*(f):[k] \Rightarrow [k'']$ by
$$\forall A \subset [k], \;  \delta _*(f)(A)=\delta\bigl(f(A)\bigr).$$
Then, every $f: [k] \Rightarrow [k']$ clearly has a unique decomposition as
$f=\delta _*(g)$ for an onto morphism $g:[k] \Rightarrow [k'']$ and a monomorphism
$\delta:[k''] \hookrightarrow [k']$ in $\Delta$.
We set:
$$\red(f):=g \quad \text{and} \quad \sup(f):=\delta.$$
\end{Not}

\subsection{The family $\bigl(W(f,\epsilon)\bigr)$ of open subsets of the simplicies}\label{definitionofW}

In the beginning of Section \ref{gencon}, we have fixed an arbitrary point $\alpha \in \Delta^n \setminus \partial \Delta^n$. 
Let us write $\alpha=\bigl(t_0(\alpha),\ldots,t_n(\alpha)\bigr)$, whence $t_i(\alpha)>0$ for every $i \in [n]$.

\begin{Def}
A family $(I_{i,j})_{0\leq i<j \leq n}$ of open intervals
of $]0,+\infty[$ with compact closure in $]0,+\infty[$ is called \defterm{$\alpha$-admissible}
when $\dfrac{t_j(\alpha)}{t_i(\alpha)}$ belongs to $I_{i,j}$ for every pair $(i,j)\in [n]^2$ such that $i<j$.
\end{Def}

Clearly, such a family exists, and we may choose one for the rest of the section.
Let then $\epsilon \in \,\left]0,1\right[$.
For any $f:[n] \Rightarrow [k]$, we define $W(f,\epsilon) \subset \Delta^k$ as the subset consisting of the points
$\beta=(t_0,\dots,t_k) \in \Delta^k$ for which
$$\forall i \in [n], \;\underset{p \in f(\{i\})}{\sum}t_p >0 \quad, \quad
\forall (i,j) \in [n]^2, \; i<j \,\Rightarrow\, \dfrac{\underset{p \in f(\{j\})}{\sum}t_p}{\underset{p \in f(\{i\})}{\sum}t_p} \in I_{i,j}
\quad \text{and} \quad \underset{p \in f([n])}{\sum}t_p > 1-\epsilon.$$
Obviously, this is an open convex subset of $\Delta^k$ and its closure is the set of those points $(t_0,\dots,t_k) \in \Delta^k$
for which
$$\forall i \in [n], \; \underset{p \in f(\{i\})}{\sum}t_p >0 \quad, \quad
\forall (i,j) \in [n]^2, \; i<j \,\Rightarrow\,   \frac{\underset{p \in f(\{j\})}{\sum}t_p}{\underset{p \in f(\{i\})}{\sum}t_p} \in \overline{I_{i,j}} \quad \text{and} \quad \underset{p \in f([n])}{\sum}t_p \geq 1-\epsilon.$$
Notice finally that $\alpha \in W(\id_{[m]},\epsilon)$.

\subsection{Completing the construction of $U_\epsilon$}

We are now almost ready to construct the family $(U_\epsilon)$ of open neighborhoods of $\pi_n(x,\alpha)$ we were looking for. 

\subsubsection{The open sets $U(f)$}\label{definitionofU(f)}

Given an onto morphism $f:[n] \Rightarrow [k]$, thus corresponding to an epimorphism
$\sigma : [k] \twoheadrightarrow [n]$, we set
$U(f):=U(\sigma) \subset A_k$.
For an arbitrary $f:[n] \Rightarrow [k]$, we set
$$U(f):=(\sup(f)^*)^{-1}(U(\red(f)) \subset A_k.$$
Obviously, $U(f)$ is an open subset of $A_k$.

\vskip 3mm
When we started working on the separation properties for simplicial spaces, one of our initial 
ideas was, for an arbitrary $\epsilon \in \,\left]0,1\right[$, to define $U_\epsilon$ as 
$$\underset{k \in \N}{\bigcup}\,\pi_k\left(  \underset{f:[n] \Rightarrow [k]}{\bigcup}U(f)\times W(f,\epsilon)\right).$$
It happens however that this does not deliver the desired construction because, in general, the family of spaces 
$\left( \underset{f:[n] \Rightarrow [k]}{\bigcup}U(f)\times W(f,\epsilon)\right)_{k \in \N}$ is not compatible. 
More precisely, conditions \eqref{faceprop} are always met by this sequence, but conditions \eqref{degprop} might not. 
To deal with that problem, we need to add an extra layer of complexity to the construction: 
this involves an ordering on the set of morphisms of the category $\Gamma '$, which we shall explain in the next paragraph. 
Afterwards, we will finally be able to give the actual definition of the $U_\epsilon$ subsets. 

\subsubsection{Ordering the morphisms of $\Gamma '$}\label{orderingsection}

Let $f:[k] \Rightarrow [k']$ and $i \in [k']$.
Set
$$\card_f(i):=\begin{cases}
\#f(\{j\}) & \text{whenever $i \in f(\{j\})$} \\
0 & \text{when $i \not\in f([k])$}.
\end{cases}
$$
\label{fplusi}
Assume now that $i<k'$. 
If $\card_f(i)=0$ and $\card_f(i+1)\geq 1$, we define
$f^{+i}$ by:
$$f^{+i}\{j\}=\begin{cases}
f(\{j\})\cup \{i\} & \text{when $i+1 \in f\{j\}$} \\
f\{j\} & \text{otherwise.}
\end{cases}$$
If $\card_f(i)\geq 1$ and $\card_f(i+1)=0$, we define
$f^{+i}$ by:
$$f^{+i}\{j\}=\begin{cases}
f(\{j\})\cup \{i+1\} & \text{when $i \in f(\{j\})$} \\
f\{j\} & \text{otherwise.}
\end{cases}$$
In any case, $f^{+i}$ is obtained from $f$ by attaching $i$ or $i+1$ to an adjacent set of the form $f\{k\}$.
We denote by $\mathcal{R}$ the binary relation defined on
$\Hom_{\Gamma  '}([k],[k'])$ by
$f \mathcal{R} f^{+i}$ for every $i$ and every $f$ for which
$f^{+i}$ is defined. We then define $\leq$ as the pre-order relation generated by $\mathcal{R}$.\label{orderdefinition}
Actually, this is an order relation on $\Hom_{\Gamma '}([k],[k'])$. Consider indeed the
order relation $\subset$ on $\Hom_{\Gamma '}([k],[k'])$
defined by
$$f \subset g \; \overset{\text{def}}{\Leftrightarrow}\; \bigl(\forall j \in [k],\; f(\{j\}) \subset g(\{j\})\bigr).$$
Notice that, whenever $f^{+i}$ is defined, $f \subset f^{+i}$.
It follows that
$f \leq g \Rightarrow f \subset g$ for every $f$ and $g$ in $\Hom_{\Gamma '}([k],[k'])$, whence $(\Hom_{\Gamma'}([k],[k']),\leq)$ is a poset
(its maximal elements are the onto morphisms). The opposite order relation will be denoted by $\geq$.

\begin{Rem}
Notice that $\leq$ is strictly stronger than $\subset$. For example, for
$f : [0] \Rightarrow [3]$ which maps $\{0\}$ to $\{0,4\}$,
and $g : [0] \Rightarrow [4]$ which maps $\{0\}$ to $\{0,2,4\}$, one obviously
has $f \subset g$ whilst the statement $f \leq g$ is false (notice that $f \subset g$ is an irreducible chain for $\subset$
whereas no $i$ satisfies $g=f^{+i}$).
\end{Rem}

\subsubsection{The definition of $U_\epsilon$}\label{uepsdef}

Let $\epsilon \in \,\left]0,1\right[$ and $k \in \mathbb{N}$. Set then
$$U_{k,\epsilon}:=\underset{f:[n] \Rightarrow [k]}{\bigcup}
\biggl(U(f) \times \underset{g \geq f}{\bigcap}W(g,\epsilon)\biggr)$$
which is clearly an open subset of $A_k \times \Delta^k$.
Set also $$U'_{k,\epsilon}:=\underset{f:[n] \Rightarrow [k]}{\bigcup}
U(f) \times \overline{W(f,\epsilon)}$$
and notice that $U_{k,\epsilon} \subset U'_{k,\epsilon}$. \\
For an arbitrary $\epsilon \in \,\left]0,1\right[$, we finally define:
$$U_\epsilon:=\underset{k \in \N}{\bigcup} \pi
_k(U_{k,\epsilon}) \subset |A|.$$
For every $\epsilon \in ]0,1[$, one has $\alpha \in
W(\id_{[m]},\epsilon)$ and $x \in U(\id_{[m]})$, whence $(x,\alpha) \in
U_{m,\epsilon}$ since $\id_{[m]}$ is maximal. This shows that $[(x,\alpha)] \in U_\epsilon$.
In the next section, we will show that $U_\epsilon$ is an open subset of $|A|$
by proving that the family $(U_{k,\epsilon})_{k \in \N}$ is \emph{compatible}.

\section{The proof that $U_\epsilon$ is an open subset of $|A|$}\label{opensec}

Here, we will prove the following proposition:

\begin{prop}\label{compatibility}
Let $\epsilon \in \,]0,1[$. Then,
\begin{enumerate}[(a)]
\item For every $k \in \mathbb{N}^*$,
$(y, \beta) \in A_k \times \Delta ^{k-1}$ and $i \in [k]$:
$$ (y,(\delta _i)_*(\beta)) \in U_{k,\epsilon} \Leftrightarrow
(d_i(y),\beta) \in U_{k-1,\epsilon}.$$
\item For every $k \in \mathbb{N}$, $(x', \alpha ') \in A_k \times \Delta ^{k+1}$ and $i \in [k]$:
$$(y,(\sigma _i)_*(\beta)) \in U_{k,\epsilon} \Leftrightarrow
(s_i(y),\beta) \in U_{k+1,\epsilon}.$$
\end{enumerate}
\end{prop}

This has the following immediate corollary as explained in the introduction of Section \ref{gencon}:

\begin{cor}
For every $\epsilon \in \,\left]0,1\right[$, the subset $U_\epsilon$ is open in
$|A|$ and $\forall k \in \N,\; U_{k,\epsilon}=\pi_k^{-1}(U_\epsilon)$.
\end{cor}

\subsection{One last notation}\label{fmoinsi}

Given $f : [k] \Rightarrow [k']$ and $i \in [k'-1]$
such that $\card_f(i) \neq 1$, we define
$f_{-i}:[k] \Rightarrow [k'-1] $ by
$$\forall j \in [k], \; f_{-i}(\{j\})=\sigma _i(f(\{j\}) \setminus \{i\}).$$
If $\card_f(k') \neq 1$, we define $f_{-k'}:[k] \Rightarrow [k'-1]$ by
$$\forall j \in [k], \; f_{-k'}(\{j\})=f(\{j\}) \setminus \{k'\}.$$
Obviously $(\delta _i)^*(f_{-i})=f$ when $\card_f(i)=0$. Furthermore, if $f \leq g$ and $f_{-i}$ is
defined, then $g_{-i}$ is defined. The following results are then straightforward:

\begin{lemme}\label{admitted1}
Let $f:[k] \Rightarrow [k']$ and $g : [k] \Rightarrow [k']$ together with some $i \in [k']$ such that $\card_f(i) \neq 1$.
Then, $$f \leq g \,\Rightarrow\, f_{-i} \leq g_{-i}.$$
\end{lemme}

\begin{lemme}\label{admitted2}
Let $f:[k] \Rightarrow [k']$ and $g:[k] \Rightarrow [k'-1]$ together with some
$i \in [k']$ such that $\card_f(i) \neq 1$.
Assume that $f_{-i} \leq g$.
\begin{itemize}
\item If $\card_g(i)=\card_g(i-1)=0$, then
$f \leq (\delta_i)_*(g)$.
\item If $\card_g(i)> 0$ and $\card_g(i-1)=0$, then
$f \leq (\delta_i)_*(g)^{+i}$.
\item If $\card_g(i)= 0$ and $\card_g(i-1)>0$, then
$f \leq (\delta_i)_*(g)^{+(i-1)}$.
\item If $\card_g(i)> 0$ and $\card_g(i-1)>0$, then
$f \leq (\delta_i)_*(g)^{+(i-1)}$ or $f \leq (\delta_i)_*(g)^{+i}$.
\end{itemize}
\end{lemme}

\begin{lemme}\label{admitted3}
Let $f: [k] \Rightarrow [k'-1]$ and $i \in [k'-1]$ be such that
$\card_f(i) >0$. Let $g \geq (\delta_i)_*(f)^{+i}$. Then, $g_{-i} \geq f$.
\end{lemme}

\subsection{Proof of statement (a) in Proposition \ref{compatibility}}

We fix an arbitrary pair $(y, \beta) \in A_k \times \Delta^{k-1}$ and an arbitrary integer $i \in [k]$.

Assume first that there exists some $f:[n] \Rightarrow [k]$
such that $$\bigl(y,(\delta_i)_*(\beta)\bigr)\in U(f) \times \underset{g\geq
  f}{\bigcap} W(g,\epsilon).$$
Hence, $(\delta_i)_*(\beta) \in W(f,\epsilon)$ which yields $\card_f(i) \neq 1$.
The rest of the proof essentially rests upon the following claim:
$$(d_i(y),\beta) \in U(f_{-i}) \times \underset{g \geq f_{-i}}{\bigcap} W(g,\epsilon).$$
\begin{itemize}
\item Firstly, we show that $d_i(U(f)) \subset U(f_{-i})$. Let $y \in U(f)$.
\begin{itemize}
\item
Assume that $\card_f(i)=0$. Obviously $\red{f_{-i}}=\red{f}$ and
$\delta _i \circ \sup(f_{-i})=\sup(f)$, whence $\sup(f_{-i})^*(d_i(y)) \in U(\red(f_{-i}))$ i.e.\ $d_i(y) \in U(f_{-i})$.
Notice conversely that, for an arbitrary $z$, the condition $d_i(z) \in U(f_{-i})$ implies $z \in U(f)$.
\item
Assume that $\card_f(i)\geq 2$. Denote respectively by $\sigma : [k'] \twoheadrightarrow
[n]$ and $\sigma ': [k'-1] \twoheadrightarrow [n]$ the epimorphisms corresponding to $\red(f)$ and $\red(f_{-i})$.
Then, there is some $j$ such that $\sigma \circ \delta _j=\sigma '$ and the square
$$\begin{CD}
[k'-1] @>{\sup(f_{-i})}>> [k-1] \\
@V{\delta _j}VV @V{\delta _i}VV \\
[k'] @>{\sup(f)}>> [k]
\end{CD}$$ is commutative.
However $d_j(\sup(f)^*(y)) \in U(\sigma ')$ since $\sup(f)^*(y) \in U(\sigma)$ and $\delta _j \in I_{\sigma}$ (cf.\ statement (c) in Proposition \ref{uproperties}). It follows that $\sup(f_{-i})^*(d_i(y)) \in U(\sigma ')$, whence $d_i(y) \in U(f_{-i})$.
\end{itemize}
\item Let now $g \geq f_{-i}$. We wish to prove that $\beta \in W(g,\epsilon)$.
\begin{itemize}
\item If $\card_g(i)=\card_g(i-1)=0$,  then the definition of the $W(h,\eta)$'s
obviously yield that the condition $(\delta _i)_*(\beta) \in W((\delta _i)_*(g),\epsilon)$ implies $\beta \in W(g,\epsilon)$, whilst Lemma \ref{admitted1} shows that $f\leq (\delta _i)_*(g)$.
\item If $\card_g(i) >0$ and $\card_g(i-1)=0$.
Then again, $(\delta _i)_*(\beta)\in W((\delta _i)_*(g)^{+i},\epsilon)$ clearly
implies $\beta \in W(g,\epsilon)$, whilst Lemma \ref{admitted2} shows that $f \leq (\delta _i)_*(g)^{+i}$.
\item If $\card_g(i)=0$ and $\card_g(i-1) >0$, then $(\delta _i)_*(\beta)
\in W((\delta _i)_*(g)^{+(i-1)},\epsilon)$ clearly implies $\beta \in W(g,\epsilon)$, whilst
Lemma \ref{admitted2} shows that $f \leq (\delta _i)_*(g)^{+(i-1)}$.
\item Either $\card _g(i-1)>0$ and $\card_g(i)>0$.
Then, both conditions $(\delta _i)_*(\beta)\in W((\delta _i)_*(g)^{+i},\epsilon)$
and $(\delta _i)_*(\beta) \in W((\delta _i)_*(g)^{+(i-1)},\epsilon)$ clearly imply that
$\beta \in W(g,\epsilon)$, whilst one has $f \leq (\delta_i)_*(g)^{+(i-1)}$ or $f \leq (\delta_i)_*(g)^{+i}$.
\end{itemize}
In any case, this shows that $\beta \in W(g,\epsilon)$, whence $\beta \in \underset{g \geq f_{-i}}{\bigcap} W(g,\epsilon)$.
\end{itemize}
Conversely, assume $\bigl(d_i(y),\beta\bigr)\in U(f) \times
\underset{g \geq f}{\bigcap} W(g,\epsilon)$ for some $f:[n] \Rightarrow [k-1]$.
It was proven earlier that if $\card_g(i)=0$, then, for every $z$, one has $z \in
U(f)$ if and only if $d_i(z) \in  U(f_{-i})$. Applying this to
$(\delta_i)_*(f)$ yields $y \in U((\delta _i)_*(f))$ since $(\delta _i)_*(f)_{-i}=f$.
On the other hand, for any $g \geq (\delta_i)_*(f)$, Lemma \ref{admitted1} shows that $g_{-i} \geq (\delta_i)_*(f)_{-i}=f$, 
whence $\beta \in W(g_{-i},\epsilon)$ and then
$(\delta_i)_*(\beta) \in W(g,\epsilon)$ by using the definition of the $W(h,\eta)$'s. We deduce that
$$\bigl(y,(\delta_i)_*(\beta)\bigr) \in U((\delta_i)_*(f)) \times
\underset{g \geq (\delta_i)_*(f)}{\bigcap} W(g,\epsilon),$$ which finishes
the proof of point (a) in Proposition \ref{compatibility}.

\begin{Rem}\label{compatadher}
A similar strategy of proof shows that for every $k\in \N^*$,
$\epsilon \in \,\left]0,1\right[$, $i \in [k]$
and $(y,\beta) \in A_k \times \Delta ^{k-1}$:
$$\bigl(y,(\delta _i)_*(\beta)\bigr) \in U'_{k,\epsilon} \,\Leftrightarrow\,
\bigl(d_i(y),\beta\bigr) \in U'_{k-1,\epsilon.}$$
\end{Rem}

\subsection{Proof of statement (b) in Proposition \ref{compatibility}}

Let $(y,\beta) \in A_k \times \Delta^{k+1}$, and $i \in [k]$.
\begin{itemize}
\item Assume that $(y,(\sigma _i)_*(\beta)) \in U(f) \times
\underset{g \geq f}{\bigcap}W(g,\epsilon)$ for some $f:[n] \Rightarrow [k]$.
\begin{itemize}
\item Assume that $\card_f(i)=0$.
We may write $\bigl(d_{i+1}(s_i(y)),(\sigma _i)_*(\beta)\bigr) \in U(f) \times
\underset{g \geq f}{\bigcap}W(g,\epsilon)$
and $(d_{i}(s_i(y)),(\sigma _i)_*(\beta)) \in U(f) \times
\underset{g \geq f}{\bigcap}W(g,\epsilon)$.
Statement (a) then yields
$(s_i(y),(\sigma _i \circ \delta _{i+1})_*(\beta))
\in U((\delta _i)_*(f)) \times
\underset{g \geq (\delta _i)_*(f)}{\bigcap}W(g,\epsilon)$
and $(s_i(y),(\sigma _i \circ \delta _i)_*(\beta))
\in U((\delta _{i+1})_*(f)) \times
\underset{g \geq (\delta _{i+1})_*(f)}{\bigcap}W(g,\epsilon)$.
However $(\delta _i)_*(f)=(\delta _{i+1})_*(f)$ since
$\card_f(i)=0$. Thus $s_i(y) \in U((\delta _i)_*(f))$,
$(\sigma _i \circ \delta _{i+1})_*(\beta) \in \underset{g \geq (\delta
  _i)_*(f)}{\bigcap}W(g,\epsilon)$ and
$(\sigma _i \circ \delta _i)_*(\beta) \in \underset{g \geq (\delta
  _i)_*(f)}{\bigcap}W(g,\epsilon)$.
However $\beta \in \left[(\sigma _i \circ \delta _i)_*(\beta),(\sigma _i \circ
\delta _{i+1})_*(\beta)\right]$ and every $W(g,\epsilon)$ is convex.
It follows that $\beta \in \underset{g \geq (\delta_i)_*(f)}{\bigcap}W(g,\epsilon)$ and finally
$(s_i(y),\beta) \in U((\delta _i)_*(f)) \times
\underset{g \geq (\delta _i)_*(f)}{\bigcap}W(g,\epsilon)$.
\item Assume that $\card_f(i) \geq 1$. Set then $f':=(\delta_i)_*(f)^{+i}$ and let $g \geq f'$.
Then, Lemma \ref{admitted3} shows that $g_{-i} \geq f$. From
$(\sigma _i)_*(\beta) \in W(g_{-i},\epsilon)$ immediately follows $\beta \in W(g,\epsilon)$ by using the definition of the
$W(h,\eta)$'s.
Furthermore, should we let $\sigma$ and $\sigma'$ denote the epimorphisms which respectively correspond to $\red(f)$ and $\red(f')$,
then there exists some $j$ such that $\sigma \circ \sigma _j =\sigma'$ and the square
$$\begin{CD}
[k'+1] @>{\sup(f')}>> [k+1] \\
@V{\sigma _j}VV @V{\sigma _i}VV \\
[k'] @>{\sup(f)}>> [k]
\end{CD}$$
is commutative. Since $\sup(f)^*(y) \in U(\red(f))$,
one has $s_j(\sup(f)^*(y)) \in U(\red(f)\circ \sigma _j)$, whence
$\sup(f')^*(s_i(y)) \in U(\red(f'))$.
Therefore, $(s_i(y),\beta) \in U(f') \times \underset{g \geq f'}{\bigcap}W(g,\epsilon)$.
\end{itemize}

\item Conversely, assume that $(s_i(y),\beta) \in U(f) \times \underset{g \leq f}{\bigcap}W(g,\epsilon)$ for some
$f:[n]\Rightarrow [k+1]$.
\begin{itemize}
\item Assume that $i \not\in f([n])$ and $i+1 \not\in f([n])$.
Then, $\card _f(i)=0$, and hence $y=d_i(s_i(y)) \in U(f_{-i})$ by the proof of statement (a).
Let $g \geq f_{-i}$. If $\card_g(i)=0$, then $(\delta _i)_*(g) \geq f$, 
whence $\beta \in W((\delta _i)_*(g),\epsilon)$ yields $(\sigma _i)_*(\beta) \in W(g,\epsilon)$.
If $\card_g(i)\geq 1$ , then $(\delta _i)_*(g)^{+i} \geq f$
and $\beta \in W((\delta _i)_*(g)^{+i},\epsilon)$ yields
$(\sigma _i)_*(\beta) \in W(g,\epsilon)$. In any case, we have shown that
$\bigl(y,(\sigma_i)_*(\beta)\bigr) \in U(f_{-i}) \times \underset{g \leq f_{-i}}{\bigcap}W(g,\epsilon)$.

\item Assume that $i$ is in $f([n])$ and $i+1$ is not.
Then, $\sup(f_{-(i+1)})=\sigma _i \circ \sup(f)$ and
$\red(f)=\red(f_{-(i+1)})$, whence $\sup(f_{-(i+1)})^*(y)=\sup(f)^*(s_i(y)) \in U(\red(f))$
i.e.\ $y \in U(f_{-(i+1)})$.
Let $g \geq f_{-(i+1)}$. Then, $(\delta _i)_*(g)^{+i} \geq f$, whilst
$\beta \in W(g,\epsilon)$ since $s_i(\beta) \in W((\delta _i)_*(g)^{+i},\epsilon)$.
It follows that $\bigl(y,(\sigma _i)_*(\beta)) \in U(f_{-(i+1)}\bigr) \times
\underset{g \geq f_{-(i+1)}}{\bigcap}W(g,\epsilon)$.

\item If $i+1$ is in $f([n])$ and $i$ is not, a similar proof as the above one shows that
$(y,(\sigma _i)_*(\beta)) \in U(f_{-i}) \times
\underset{g \geq f_{-i}}{\bigcap}W(g,\epsilon)$.

\item Assume that both $i$ and $i+1$ belong to $f([n])$. We claim that
$\{i,i+1\} \subset f\{j\}$ for some $j \in [m]$.
Assuming this holds, then $y=d_i(s_i(y)) \in U(f_{-i})$ since $\card_f(i) \geq 2$,
whilst, for every $g \geq f_{-i}$, one has $(\delta _i)_*(g)^{+i}\geq f$, therefore
$\beta \in W((\delta _i)_*(g)^{+i},\epsilon)$ i.e.\ $(\sigma _i)_*(\beta) \in W(g,\epsilon)$.
This shows that $\bigl(y,(\sigma _i)_*(\beta)\bigr) \in U(f_{-i}) \times \underset{g \geq f_{-i}}{\bigcap} W(g,\epsilon)$.

Let us finally prove the above claim. Assume indeed that no $j \in [m]$ satisfies
$\{i,i+1\} \subset f\{j\}$, and let $\sigma : [k'] \twoheadrightarrow [n]$
be the epimorphism corresponding to $\red(f)$.
Then, for some one-to-one morphism $\delta : [k] \hookrightarrow [k']$ and for some $j' \in [k'-1]$,
the square
$$\begin{CD}
[k'] @>{\sup(f)}>> [k] \\
@V{\sigma _{j'}}VV @V{\sigma _i}VV \\
[k'] @>{\delta}>> [k]
\end{CD}$$
is commutative with $\sigma (j') \neq \sigma (j'+1)$.
Then, $(\sigma_i \circ \sup(f))^*(y) \in U(\red(f))$ yields
$s_{j'}(\delta ^*(y)) \in U(\sigma)$, which contradicts property (e)
of Proposition \ref{uproperties} applied to $\tau=\sigma_{j'}$.
\end{itemize}
In any case, we have shown that $\bigl(y,(\sigma _i)_*(\beta)\bigr) \in U_{k,\epsilon.}$
\end{itemize}

This finishes the proof of Proposition \ref{compatibility}.

\section{Application to simplicial Hausdorff spaces}\label{sepsec}

In this section, we assume that $A$ is a simplicial Hausdorff space.
We arbitrarily choose two pairs $(x,\alpha)$ and $(y,\beta)$ such that
$x\in A_n$ is non-degenerate, $y\in A_m$ is non degenerate,
$\alpha \in \Delta ^n \setminus \partial \Delta ^n$ and
$\beta \in \Delta ^m \setminus \partial \Delta ^m$.
We also pick an arbitrary integer $N \geq \max(n,m)$. Example \ref{exvois} yields that we may choose
an $x$-admissible family
$(U_\sigma)_{\sigma : [N] \twoheadrightarrow [n]}$ and a $y$-admissible family
$(V_\sigma)_{\sigma : [N] \twoheadrightarrow [n]}$.
Using the procedure of Sections \ref{upar} and \ref{suitfam}, we recover two families
$(U(f))$ and $(V(g))$.

We also choose an $\alpha$-admissible family
$(I_{i,j})$ of intervals and a $\beta$-admissible family of intervals. For every $\epsilon \in \,\left]0,1\right[$, we obtain
respective families $(W(f,\epsilon))_{f:[n] \Rightarrow [k]}$ and $(T(f,\epsilon))_{f:[n] \Rightarrow [k]}$ corresponding
to $(I_{i,j})$ and to $(J_{k,l})$.

For every $k \in \N$, the procedure of Section \ref{uepsdef} yields subsets $U_{k,\epsilon}$ and
$U'_{k,\epsilon}$ of $A_k \times \Delta^k$ from $(U(f))$ and $(W(f,\epsilon))$,
and subsets $V_{k,\epsilon}$ and
$V'_{k,\epsilon}$ of $A_k \times \Delta^k$ from $(V(g))$ and $(T(g,\epsilon))$.
This yields open subsets $U_\epsilon$ and $V_\epsilon$ of $|A|$.

\subsection{A sufficient condition for disjointness}

\begin{prop}\label{separationprop}
Assume that for every $k\in \N$, every onto
$f:[n] \Rightarrow [k]$, every onto
$g:[m] \Rightarrow [k]$, and every
$\epsilon \in \,\left]0,1\right[$, one has
$$\Bigl(U(f) \times \overline{W(f,\epsilon)}\Bigr) \cap
\Bigl(V(g) \times \overline{T(g,\epsilon)}\Bigr)=\emptyset.$$
Then, there exists an $\eta \in \,\left]0,1\right[$ such that $U_\eta \cap V_\eta =\emptyset$.
\end{prop}

\begin{proof}
Notice first that, for every $0<\epsilon \leq \eta <1$ and every $k \in \N$, one has
$U_{k,\epsilon} \subset U'_{k,\epsilon} \subset U'_{k,\eta}$ and $V_{k,\epsilon} \subset V'_{k,\epsilon} \subset
V'_{k,\eta}$. It thus suffices to provide some $\eta \in \,]0,1[$ such that $U'_{k,\eta} \cap V'_{k,\eta}=\emptyset$.
There are two major steps in the proof of the existence of $\eta$:
\begin{enumerate}[(1)]
\item By a finite induction process, we will show that there is some $\eta \in \,\left]0,1\right[$ such that
$$\forall k \leq 2(n+2)(m+2), \quad U'_{k,\eta} \cap V'_{k,\eta}=\emptyset.$$
\item By another induction process, we will show that such an $\eta$ necessarily satisfies:
$$\forall k \geq 2(n+2)(m+2), \quad U'_{k,\eta} \cap V'_{k,\eta}=\emptyset.$$
\end{enumerate}
We first consider the case $k=0$. If $n >0$ (resp.\ $m>0$) then
$U'_{k,\epsilon}=\emptyset$ (resp.\ $V'_{k,\epsilon}=\emptyset$).
If $m=n=0$, then $U'_{0,\epsilon}=U_{\id_{[0]}}$ and $V'_{0,\epsilon}=V_{\id_{[0]}}$ are
clearly disjoint for every $\epsilon$.
Let now $M$ be an integer such that $0 \leq M< 2(n+2)(m+2)$ and there is
some $\epsilon_M \in \,\left]0,1\right[$ such that
$U_{k,\epsilon _M} \cap V_{k,\epsilon _M}=\emptyset$ whenever $k \leq M$.
Let $f:[n] \Rightarrow [M+1]$ and $g:[m] \Rightarrow [M+1]$.
Clearly, we may assume that either $f$ or $g$ is not onto and use a \emph{reductio ad absurdum} by assuming that
$(U(f) \times \overline{W(f,\epsilon)}) \cap (V(g) \times \overline{T(g,\epsilon)}) \neq \emptyset$
for every $\epsilon \in \,\left]0,1\right[$. We may then choose some
$y \in U(f) \cap V(g)$ and, for every $i \in \N^*$, some $\alpha_i \in \overline{W(f,\frac{1}{i}))} \cap
\overline{T(g,\frac{1}{i})}$.
Since $\overline{W(f,\epsilon _M)} \cap \overline{T(g,\epsilon _{M})}$ is closed in $\Delta ^{M+1}$, and
therefore compact, the sequence $(\alpha_i)$ has an adherence value
$\overline{\alpha} \in \overline{W(f,\epsilon _M)} \cap
\overline{T(g,\epsilon _{M})}$. If $f$ is not onto, then picking some $j \in [M+1] \setminus f([n])$, we see that
$t_j(\alpha_n) \leq \frac{1}{n}$ for any $n\in \N^*$, whence $t_j(\bar{\alpha})=0$. In any case,
we see that $\overline{\alpha} \in \partial \Delta^{M+1}$. \\
We may thus write $\overline{\alpha}=\delta_i^*(\overline{\beta})$ for some $i$ and some $\overline{\beta} \in \Delta^M$.
Therefore $(y,(\delta_i)^*(\bar{\beta}))
\in U'_{M+1,\epsilon _M} \cap V'_{M+1,\epsilon _M}$, which implies, using Remark
\ref{compatadher}, that $(d_i(y),\bar{\beta}) \in U'_{M,\epsilon _M} \cap V'_{M,\epsilon _M}=\emptyset$. This is a contradiction,
whence the existence of some $\epsilon _{f,g} \in \,\left]0,1\right[$ such that
$$\bigl(U(f) \times \overline{W(f,\epsilon_{f,g})}\bigr)
\cap \bigl(V(g) \times \overline{T(g,\epsilon_{f,g})}\bigr) = \emptyset.$$
Setting $\epsilon'_{f,g}=\min(\epsilon _M,\epsilon_{f,g})$,
and then $\epsilon _{M+1}=\min \{\epsilon _{f,g} \quad \text{s.t. $f$ and
  $g$ are not both onto}\}$, we see that $\epsilon _{M+1} \leq \epsilon _M$, and $U'_{M+1,\epsilon _{M+1}} \cap
V'_{M+1,\epsilon _{M+1}}=\emptyset$, whence
$$\forall k \leq M+1, \;U'_{k,\epsilon _{M+1}} \cap V'_{k,\epsilon_{M+1}}=\emptyset.$$
This finite induction process yields some $\eta \in \,\left]0,1\right[$ such that
$$\forall k \leq 2(n+2)(m+2),\; U'_{k,\eta} \cap V'_{k,\eta}=\emptyset.$$
We now move on to step (2). Let then $M \geq 2(n+2)(m+2)$ be such that
$U'_{M,\eta} \cap V'_{M,\eta}=\emptyset$.
Let $f:[n] \Rightarrow [M+1]$ and $g:[m] \Rightarrow [M+1]$.
We will prove that
$$(U(f) \times \overline{W(f,\eta))} \cap (V(g) \times
\overline{T(g,\eta)})=\emptyset$$
by a \emph{reductio ad absurdum}. Assume that there is some
$(z,\gamma) \in (U(f) \times \overline{W(f,\eta)}) \cap (V(g) \times \overline{T(g,\eta)})$.
\begin{itemize}
\item Assume that there exists $i \in [n]$ and $j \in [m]$ such that
$\card(f(\{i\}) \cap g(\{j\})) \geq 2$, and choose two distinct elements $k<l$ in $f(\{i\}) \cap
g(\{j\})$. Define $\gamma ' \in \Delta^M$ by:
$$t_p(\gamma')=\begin{cases}
t_p(\gamma) & \text{for $p<k$} \\
t_k(\gamma)+t_l(\gamma) & \text{for $p=k$} \\
t_p(\gamma) & \text{for $k<p<l$} \\
t_{p+1}(\gamma) & \text{for $l \leq p \leq M$.}
\end{cases}$$
Then, $(\delta _l)_*(\gamma ') \in \overline{W(f,\eta)}
\cap \overline{T(g,\eta)}$, whence $(z,(\delta _l)_*(\gamma ')) \in U_{N+1,\eta}' \cap V_{N+1,\eta}'$ \\
i.e.\ $(d_l(z),\gamma ') \in U_{N,\eta}' \cap V_{N,\eta}'=\emptyset$ (see again Remark \ref{compatadher}), which is a contradiction.
\item Assume now that, for every $i \in [n]$ and $j \in [m]$,
$\card(f(\{i\}) \cap g(\{j\})) \leq 1$.
\begin{itemize}
\item Assume further that
$f([n]) \cup g([m]) \neq [M+1]$, and choose $k\in [M+1] \setminus \bigl(f[n] \cup g([m])\bigr)$.
Define then $\gamma ' \in \Delta^M$ by:
$$t_p(\gamma')=\begin{cases}
\frac{t_p(\gamma)}{1-t_k(\gamma)} & \text{for $p<k$} \\
\frac{t_{p+1}(\gamma)}{1-t_k(\gamma)} & \text{for $p\geq k$.}
\end{cases}$$
Then, $(\delta _k)_*(\gamma ') \in \overline{W(f,\eta)} \cap \overline{T(g,\eta)}$, i.e.\
$(z,(\delta _k)_*(\gamma ')) \in U'_{M+1,\eta} \cap V'_{M+1,\eta}$
and again Remark \ref{compatadher} yields the contradiction $(d_k(z),\gamma ') \in U'_{M,\eta} \cap V'_{M,\eta}=\emptyset$.

\item Assume finally that $f([n]) \cup g([m]) = [M+1]$, e.g.\
$\card(f([n])) \geq \card(g([m]))$, whence $\card(f([n])) \geq (n+2)(m+2)$ and we may find some $i \in [n]$
such that $\card(f(\{i\})) \geq m+3$. Since we have assumed that
$\forall j \in [m], \; \card(f(\{i\}) \cap g(\{j\})) \leq 1$, we deduce that
$$\card (f(\{i\}) \setminus g([m])) \geq 2.$$
We then choose two distinct elements $k<l$ of $f(\{i\}) \setminus g([m])$ and define $\gamma ' \in \Delta^M$ by :
$$t_p(\gamma')=\begin{cases}
t_p(\gamma) &  \text{for $p<k$} \\
t_k(\gamma)+t_l(\gamma) & \text{for $p=k$}  \\
t_p(\gamma) & \text{for $k<p<l$} \\
t_{p+1}(\gamma) & \text{for $l \leq p\leq N$.}
\end{cases}$$
Then, $(\delta _l)^*(\gamma ') \in \overline{W(f,\eta)} \cap \overline{T(g,\eta)}$, 
whence $(z,(\delta _l)^*(\gamma ')) \in U_{M+1,\eta}' \cap V_{M+1,\eta}'$ and Remark \ref{compatadher} yields the final contradiction
$(d_l(z),\gamma ') \in U_{M,\eta}' \cap V_{M,\eta}'=\emptyset$.
\end{itemize}
\end{itemize}
This shows that $U'_{M+1,\eta} \cap V'_{M+1,\eta}=\emptyset$. This induction process then proves that
$\forall k \in \N, \; U'_{k,\eta} \cap V'_{k,\eta}=\emptyset$ and it follows that $U'_\eta \cap V'_\eta =\emptyset$.
\end{proof}

\subsection{Finishing the proof of Theorem \ref{main}}

Theorem \ref{main} is now within our reach. We now assume
$(x,\alpha) \neq (y,\beta)$. We will then show that the various
admissible families that were used in the construction of the open subset $U_\epsilon$ and $V_\epsilon$
may be carefully chosen so that the assumptions of Proposition \ref{separationprop} are satisfied.
We will need to tackle separately the case $x \neq y$ and the case $x=y$.
The following basic lemma on simplicial sets will prove essential:

\begin{lemme}\label{simpset}
Let $N \geq \max(m,n)$.
\begin{enumerate}[(i)]
\item If $x \neq y$, then $\sigma^*(x) \neq \tau^*(y)$ for every $\sigma : [N] \twoheadrightarrow [n]$ and
$\tau : [N] \twoheadrightarrow [m]$.
\item If $x=y$, then  $\sigma^*(x) \neq \tau^*(y)$ for every $\sigma : [N] \twoheadrightarrow [n]$ and
$\tau : [N] \twoheadrightarrow [m]$ with $\sigma \neq \tau$. 
\end{enumerate}
\end{lemme}

\begin{proof}
Assume that there is some $\sigma : [N] \twoheadrightarrow [n]$ and some
$\tau : [N] \twoheadrightarrow [m]$ such that
$\sigma^*(x)=\tau^*(y)$. Applying Lemma \ref{degenlemma} to both $x$ and $y$, we easily find that
$m=n$ and $\sigma=\tau$. However, $\sigma^*$ is one-to-one since $\sigma$ has a section in $\Delta$.
Hence, $x=y$, which proves both statements in the lemma.
\end{proof}

\subsubsection{The case $x \neq y$}

Since $A_{m+n}$ is Hausdorff, the above lemma and the method of
Example \ref{exvois} shows that we may choose the families
$(U_\sigma)_{\sigma : [N] \twoheadrightarrow [n]}$ and  $(V_\tau)_{\tau : [N] \twoheadrightarrow [m]}$ so that
$$\forall (\sigma, \tau), \; U_\sigma \cap V_\tau =\emptyset.$$
The following lemma will then show that we may use Proposition \ref{separationprop},
which will complete the case $x \neq y$:

\begin{lemme}\label{uetvsepares}
\begin{enumerate}[(a)]
\item For every $k \leq m+n$ and every $\sigma : [k] \twoheadrightarrow
  [n]$ and $\tau : [k] \twoheadrightarrow [m]$, one has $U_\sigma \cap V_\tau =\emptyset$.
\item For every $k \in \mathbb{N}$, every $\sigma : [k] \twoheadrightarrow [n]$ and
$\tau : [k] \twoheadrightarrow [m]$, one has $U(\sigma) \cap V(\tau) =\emptyset$.
\end{enumerate}
\end{lemme}

\begin{proof}
\begin{enumerate}[(a)]
\item Choose an arbitrary $\sigma ':  [n+m] \twoheadrightarrow [k]$. Then, the assumptions show that
$(\sigma ')^*(U_\sigma \cap V_\tau) \subset U_{\sigma \circ \sigma '}
  \cap V_{\tau \circ \sigma '}=\emptyset$.
\item If $k \leq m+n$, then $U(\sigma) \subset U_\sigma$, and
  $V(\tau) \subset V_\tau$, whence we may readily use (a).
In the case $k \geq m+n$, we proceed by onward induction. Let $k\geq m+n$ be such that
$U(\sigma) \cap V(\tau)=\emptyset$ for every $\sigma : [k] \twoheadrightarrow [n]$ and $\tau:[k] \twoheadrightarrow [m]$. \\
Let then $\sigma : [k+1]
  \twoheadrightarrow [n]$ and $\tau:[k+1] \twoheadrightarrow [m]$.
We set $\calJ_\sigma := \{ i \in [k+1] : \sigma \circ
  \delta _i : [k] \twoheadrightarrow [n] \}$ and define $\calJ_\tau$ accordingly.
Then, $\# \calJ_\sigma \geq k+2-n$ and $\#\calJ_\tau \geq k+2- m$. Thus
$\#\calJ_\sigma + \#\calJ_\tau > k+2$, and so
$\calJ_\sigma \cap \calJ_\tau \neq \emptyset$. \\
Choose $i \in \calJ_\sigma \cap \calJ_\tau$. Point (c) in Proposition \ref{uproperties} and
the induction assumption then show that $d_i(U(\sigma) \cap V(\tau)) \subset U(\sigma \circ \delta _i)
  \cap V(\tau \circ \delta _i)=\emptyset$, whence $U(\sigma) \cap V(\tau) =\emptyset$.
\end{enumerate}
\end{proof}

\subsubsection{The case $x=y$}

We now set $N:=(n+1)^2$. Again, Lemma \ref{simpset} and the method of
Example \ref{exvois} show we may choose the $x$-admissible family
$(U_\sigma )_{\sigma : [(n+1)^2] \twoheadrightarrow [n]}$ so that
$ U_\sigma \cap U_\tau =\emptyset$ whenever $\sigma \neq \tau$.

For $i \in [n]$, set $\delta(i):=t_i(\alpha)-t_i(\beta)$.
Since $\alpha \neq \beta$, and $\underset{0\leq i \leq n}{\sum}\delta(i)=0$, we may find two indices
$(k,l) \in [n]^2$ such that $\delta(k) <0<\delta(l)$, e.g.\ with $k<l$ (if not, we simply switch $(x,\alpha)$ and $(y,\beta)$).
Clearly, we may choose our $\alpha$-admissible family $(I_{i,j})$ and our $\beta$-admissible family $(J_{i,j})$
so that $\overline{I_{k,l}} \cap \overline{J_{k,l}} =\emptyset$.
Obviously
$$\forall \epsilon \in \,\left]0,1\right[, \overline{W(\id_{[n]},\epsilon)} \cap \overline{T(\id_{[n]},\epsilon)} =\emptyset.$$
Proposition \ref{separationprop} may then be used thanks to the next lemma, and this will complete both the case $x=y$ and the proof
of Theorem \ref{main}:

\begin{lemme}Let $\epsilon \in \,\left]0,1\right[$.
\begin{enumerate}[(a)]
\item For every $k \leq (n+1)^2$, every onto morphism
$f:[n] \Rightarrow [k]$ and every onto morphism $g: [n] \Rightarrow [k]$, one has
$$f \neq g \,\Rightarrow\, U(f) \cap U(g) =\emptyset$$
\item For every $k \in \mathbb{N}$, every onto morphism
$f:[n] \Rightarrow [k]$ and every onto morphism $g: [n] \Rightarrow [k]$,
$$\bigl(U(f) \times \overline{W(f,\epsilon)}\bigr)
\cap \bigl(U(g) \times \overline{T(g,\epsilon)}\bigr)=\emptyset$$
\end{enumerate}
\end{lemme}

\begin{proof}
\begin{enumerate}[(a)]
\item Obviously, it suffices to prove that
$U_\sigma \cap U_\tau =\emptyset$ whenever $\sigma \neq \tau$.
Let then $\sigma : [k] \twoheadrightarrow [n]$ and
$\tau : [k] \twoheadrightarrow [n]$, with $\sigma \neq \tau$.
Choose some $\sigma ':[(n+1)^2] \twoheadrightarrow [k]$.
Then, $(\sigma ')^*(U_\sigma \cap U_\tau) \subset U_{\sigma \circ \sigma '}
\cap U_{\tau \circ \sigma '}$.
However, if $\sigma \circ \sigma '=\tau \circ \sigma '$, then $\sigma
=\tau$. Thus $\sigma \circ \sigma '\neq \tau \circ \sigma '$
and we deduce from the hypothesis that
$U_{\sigma \circ \sigma '} \cap U_{\tau \circ \sigma '} =\emptyset$, which yields $U_\sigma \cap U_\tau =\emptyset$.

\item Let $k \leq (n+1)^2$. Clearly
$\overline{W(f,\epsilon)} \cap \overline{T(f,\epsilon)}
=\emptyset$ given the construction of the families $(I_{i,j})$ and $(J_{i,j})$.
Then, either $f \neq g$ and $U(f) \cap U(g) =\emptyset$, or $f=g$ and
$\overline{W(f,\epsilon)} \cap \overline{T(g,\epsilon)} =\emptyset$ since
$\overline{I_{k,l}} \cap \overline{J_{k,l}} =\emptyset$. In any case, the claimed property is proven. \\
For the case $k \geq (n+1)^2$, we proceed by onward induction.
Let $k\geq (n+1)^2$ be such that the claimed result holds for every pair of onto morphisms from
$[n]$ to $[k]$ in $\Gamma'$.
Let $f:[n] \Rightarrow [k+1]$ and
$g:[n] \Rightarrow [k+1]$ be onto morphisms, and
$\sigma : [k+1] \twoheadrightarrow [n]$ and
$\tau : [k+1] \twoheadrightarrow [n]$ the epimorphisms in $\Delta$ respectively associated to them. \\
Assume that $\# f(\{i\}) \cap \#g(\{j\}) \leq 1$ for every $(i,j) \in [n]^2$.
Then, $\# f(\{i\}) \leq n+1$ for every $i$ (since $g$ is onto), and $k+1=\# f([n]) \leq
(n+1)^2$ which contradicts the fact that $f$ is onto.
Hence, we may find a pair $(i,j) \in [n]^2$ and some $l \in [k+1]$ such that
$\{l,l+1\} \subset f(\{i\}) \cap g(\{j\})$.
Assume finally that there is some $(z,\gamma)$ in $(U(f) \times \overline{W(f,\epsilon)}) \cap (U(g) \times
\overline{T(g,\epsilon)})$. Notice then that
$$\bigl(d_l(z),(\sigma _l)_*(\gamma)\bigr) \in
\bigl(U(f_{-l}) \times \overline{W(f_{-l},\epsilon)}\bigr) \cap \bigl(U(g_{-l}) \times
\overline{T(g_{-l},\epsilon)}\bigr).$$
Indeed, it is obvious on the one hand that
$(\sigma _l)_*(\gamma) \in \overline{W(f_{-l},\epsilon)} \cap \overline{T(g_{-l},\epsilon)}$;
on the other hand $\sigma \circ \delta_l$ and $\tau \circ \delta_l$ clearly are the epimorphisms respectively associated to the onto
morphisms $f_{-l}$ and $g_{-l}$, whence point (c) in Lemma \ref{uproperties} shows that
$d_l(z) \in U(f_{-l}) \cap U(g_{-l})$.
Finally, the contradiction comes from the induction hypothesis since $f_{-l}$ and $g_{-l}$ are both onto.
\end{enumerate}
\end{proof}

\bibliographystyle{plain}

\end{document}